\documentclass{nzjm}
\usepackage{amssymb,mathrsfs}
\usepackage{amsmath}
\usepackage[normalem]{ulem}
\usepackage{csquotes}
\usepackage{hyperref}
\usepackage{amssymb,amstext,titletoc,fancyhdr,color,xcolor,calc,graphicx}

\def\currentvolume{53}
\def\currentyear{2022}
\def\pagerange{63--79}
\def\doinumber{187}
\newtheorem{theorem}{Theorem}[section]
\newtheorem{lemma}[theorem]{Lemma}
\newtheorem{corollary}[theorem]{Corollary}

\theoremstyle{definition}
\newtheorem{definition}[theorem]{Definition}

\theoremstyle{remark}
\newtheorem{remark}[theorem]{Remark}

\numberwithin{equation}{section}

\def\Res{\mathop{\rm Res}}

\begin{document}

\title[$D$-module approach to Liouville's Theorem]{$D$-module approach to Liouville's Theorem for difference operators}
\author[K. H. Cheng]{Kam Hang Cheng}
\author[Y. M. Chiang]{Yik Man Chiang}
\author[A. Ching]{Avery Ching}
\date{20 November, 2021}
\thanks{The first author was partially supported by a fellowship award from the Research Grants Council of the Hong Kong Special Administrative Region, China (Project No. HKUST PDFS2021-6S04). The second author was partially supported by the Research Grants Council of Hong Kong (GRF600609).}

\begin{abstract}
We establish analogues of Liouville's theorem in the complex function theory, with the differential operator replaced by various difference operators. This is done generally by the extraction of (formal) Taylor coefficients using a residue map which measures the obstruction having local ``anti-derivative". The residue map  is based on a Weyl algebra or $q$-Weyl algebra structure satisfied by each corresponding operator. This explains the different senses of ``boundedness" required by the respective analogues of Liouville's theorem in this article.
\end{abstract}

\maketitle

\section{Introduction}
Liouville's theorem from classical complex function theory continues to enjoy new generalizations and interpretations in a wide range of mathematical disciplines, see e.g. \cite{Chu_Vu_2003, Farina}. In this article, however, we return to its origin from a different perspective. 
It is well-known that the classical Liouville theorem  for an entire function $f$ can be derived from the vanishing of the coefficients in its Taylor series expansion
    \[
       f(x)= \sum_{k=0}^\infty a_k x^k
    \]
under the boundedness assumption. In view of the different versions of Little Picard's theorem on meromorphic functions with respect to various difference operators discovered in \cite{BHMK, Halburd-Korhonen,Cheng-Chiang, Chiang-Feng3}, it is natural to ask if there are corresponding Liouville-type theorems for these difference operators.  We formulate a general residue theory in a Weyl algebraic setting that allows us to consider Liouville's theorem for the (forward) difference operator $\Delta f(x)=f(x+1)-f(x)$. Our theory can also handle the backward difference operator $\nabla f(x)=f(x)-f(x-1)$, the $q$-difference operator $D_qf(x)=\big(f(x)-f(x/q)\big)/(x-x/q)$ (\cite{Annaby-Mansour}, \cite{Etingof}, \cite[p. 196]{Ismail}), and of course, the differential operator $Df(x)=f^\prime(x)$.

Here the Taylor series expansions are replaced by expansions in interpolation polynomials appropriate to their respective difference operators. In the case of difference operator $\Delta$, they are replaced by Newton series expansions
    \begin{equation}\label{E:newton}
        \sum_{k=0}^\infty a_k x(x-1)\cdots (x-k+1),
    \end{equation}
which converge in an appropriate half-plane or possibly the whole plane \cite[pp. 163-168]{Gelfond}.

What are the natural growth restrictions that lead to the vanishing of the Newton expansion coefficients?

It turns out that this problem is related, except in historical interests \cite{Koppelman,Laita,Synowiec}, to what appears to be a forgotten work of Boole \cite{Boole} amongst many British mathematicians \cite[p. 157]{Koppelman} about series solutions to differential and finite difference equations which is entirely formal. Boole noticed that there is a unified,  formal algebraic framework underlying these series solutions. Indeed Boole \cite{Boole} suggested to consider two operators $\rho$ and $\pi$ which satisfy the relation
    \begin{equation}
        \label{E:boole}
                  \rho f(\pi) u =\lambda f(\pi)  \rho u
    \end{equation}
for every function $u$ on which those operators act, and every polynomial $f(\cdot)$; and $\lambda$ is another operator acting on $f$, e.g. $\lambda f(\pi)=f(\phi(\pi))$ for some $\phi$. This important relation about the syntax of $\rho$ and $\pi$ is generally being regarded as a prelude to Boole's foundational work in symbolic logic \cite{Koppelman, Laita, Synowiec, Mansour_Schork}. However, the relation \eqref{E:boole} can also be regarded, from today's viewpoint, as a kind of ``generalised commutator relation" involving the symbols $\rho$ and $\pi$.  Applying two different semantics to $\rho$ and $\pi$, Boole successfully solved, amongst others, for series solutions to confluent hypergeometric type differential equations \cite[p. 236, Ex. 1]{Boole} and difference equations under the same syntactic context. Boole's approach gives, to the best of authors' knowledge, the first unifying treatment for different types of series expansions mentioned above.

For different choices of the $\lambda$ that appears in \eqref{E:boole} (which ``quantifies" the relationship between $\rho$ and $\pi$) against their corresponding (e.g. differential/difference) operators, we provide a unifying modern approach using Weyl algebra and its modifications, including $q$-Weyl algebras. In this approach, when we work with the \textit{Weyl algebra} 
	\[
		\mathcal{A}=\mathbb{C}\langle X,\partial\rangle /\langle \partial X-X\partial-1\rangle,
	\]
	a usual entire function $g$ having Taylor series expansion
    \[
        g(x)=\sum_{k=0}^\infty a_kx^k
    \]
    is identified as an element
    \[
    	G=\sum_{k=0}^\infty a_kX^k
    \]
    in the left $\mathcal{A}$-module $\overline{\mathcal{A}/\mathcal{A}\partial}$ in an $X$-adic completion \cite{CCT}; and the derivative \linebreak $g'(x)=\sum_{k=1}^\infty a_kkx^{k-1}$ will then be correspondingly identified as the element $\partial G=\partial\sum_{k=0}^\infty a_kX^k=\sum_{k=1}^\infty ka_kX^{k-1}$ in $\overline{\mathcal{A}/\mathcal{A}\partial}$. Here the symbol $X$ need not be interpreted as the usual monomial, but instead, the pair of symbols $X$ and $\partial$ play respectively the roles of $\rho$ and $\pi$ in \eqref{E:boole}.  In the classical setting, the extraction of the Taylor coefficients $b_k$ in the proof of Liouville's theorem was done by Cauchy's contour integrals.  In the current setting, the extraction of the coefficients $b_k$ is done by measuring the obstruction to the existence of ``antiderivatives" of $G/X^{k+1}$ using the map $\mathrm{Res}$ in the exact sequence
        \[
            \overline{\mathcal{A}/\mathcal{A}\partial}
            \stackrel{\partial\times}{\longrightarrow}
            \overline{\mathcal{A}/\mathcal{A}\partial}
            \stackrel{\mathrm{Res}}{\longrightarrow}
            \mathbb{C}
            \longrightarrow 0.
        \]
The criteria of vanishing of the coefficients will then depend on the map $\mathrm{Res}$ above. In the classical setting, this map is presented as a Cauchy integral and the usual boundedness condition will suffice; in the setting that is appropriate for Newton series expansions, it turns out that one can present this map $\mathrm{Res}$ as a Barnes-type integral, so one needs a growth restriction along infinite vertical straight lines, and a ``boundedness" condition in direction parallel to the positive real axis, in order to make the coefficients vanish. Different choices of the maps $\mathrm{Res}$ with respect to their corresponding operators would yield different ``boundedness" conditions.

\medskip
This article is organized as follows. In \S\ref{S:main}, we first introduce the Weyl-algebraic framework of our theory, together with a realization of it in the context of difference operator. We start with some preliminary constructions with the Weyl algebra $\mathcal{A}$ in \S\ref{SS:Weyl}, which will lead to various kinds of series expansions for entire functions when different left $\mathcal{A}$-modules are considered. In \S\ref{SS:residues}, the construction will be extended by allowing $1/X$ to come into play, so that one can handle Laurent-type series expansions and talk about the notion of residues. In \S\ref{SS:DLiouville}, a difference operator analogue of Liouville's theorem will be proved. The machinery developed so far will help in extracting the coefficients in the series expansions. Following the same structure of \S\ref{S:main}, the $q$-difference operator will then be investigated in \S\ref{S:q}, and the classical situation regarding the differential operator will be handled briefly in \S\ref{S:classical}. Finally in \S\ref{S:discussion}, we will 
review our main results, give an application of our difference Liouville's theorem to strengthen a theorem by de Branges \cite{deBranges} about the periodicity of a Hilbert space of entire functions defined by him \cite{deBranges_1968}, and to
discuss on possible further generalizations of our theory in special functions.

\medskip
Throughout this article, we adopt the following notations:
\begin{enumerate}
	\item $\mathbb{C}^*$ denotes the punctured complex plane $\mathbb{C}\backslash\{0\}$,
	\item For each open set $U\subset\mathbb{C}$, $\mathcal{O}(U)$ denotes the space of all analytic functions defined on $U$.
	\item For each $r>0$ and each entire function $f$, the maximum modulus of $f$ on the circle of radius $r$ centered at $0$ is denoted by
	\[
		M(r,f)=\max_{|z|=r}{|f(z)|}.
	\]
\end{enumerate}

\section{Main results: Difference Liouville Theorem}
\label{S:main}
\subsection{The Weyl algebra}
\label{SS:Weyl}

We start by introducing the notion of the Weyl algebra. It is the following non-commutative algebra generated by two symbols which we call $X$ and $\partial$. It contains all the essential algebraic information of polynomial bases of various kinds.

\begin{definition}\label{D:weyl-alg}
The \textit{Weyl algebra} $\mathcal{A}$ is the (non-commutative) $\mathbb{C}$-algebra defined by
	\[
\mathcal{A}=\mathbb{C}\langle X,\partial\rangle/\langle \partial X-X\partial-1\rangle.
\]
\end{definition}
\bigskip

Fix $L\in\mathcal{A}$, and denote by $\mathcal{A}L$ the left ideal generated by $L$. Then the quotient
\[
\mathcal{A}/\mathcal{A}L
\]
is a left $\mathcal{A}$-module. In this article, we focus on the simplest case when $L=\partial$. Then the left $\mathcal{A}$-module $\mathcal{A}/\mathcal{A}\partial$ becomes the ``universal model" for the various spaces of analytic functions each of which is equipped with a specific operator. It is of course possible to study such a quotient with other choices of $L$, but we will leave this to future works.

\bigskip
In order to handle series expansions of entire functions, it is not enough to just consider the quotient $\mathcal{A}/\mathcal{A}\partial$, but to enlarge $\mathcal{A}$ so that the quotient above contains not only polynomials in $X$ but also power series in $X$.

\begin{definition}\label{D:AF}
Fix a $\mathbb{C}$-vector subspace $\mathbb{C}[X]\subset F\subset\mathbb{C}[[X]]$ of the space of formal power series in $X$.  We use the notation $\mathcal{A}_F$ to denote the ring extension of $\mathcal{A}$ generated by $F$ and $\partial$ as a $\mathbb{C}$-algebra, subject to the relation $\partial X-X\partial-1$. We suppress the emphasis of the dependence of $\mathcal{A}_F$ on $F$ in later usages when there is no ambiguity.
\end{definition}

As we will see later in \S\ref{S:classical}, when $\partial$ is interpreted as the usual differential operator $D$, the above set-up recovers the Maclaurin expansions (and hence the classical Liouville theorem). In fact, the above set-up is completely general, and a slight modification of the algebra $\mathcal{A}$ would allow us to handle the \textit{Jackson} $q$-difference operator \cite[p. 196]{Ismail}, to be considered in \S\ref{S:q}.

\subsubsection{Application to the difference operator $\Delta$} \label{SS:Difference}

Recall from \cite[pp. 163-168]{Gelfond} (see also \cite[p. 109]{Noerlund}) that (i) a Newton series \eqref{E:newton}
converges uniformly on every compact subset of $\mathbb{C}$ if
\[
	\lim_{n\to\infty}\frac{\ln |\sum_{k=n}^\infty(-1)^k k!\,a_k|}{\ln n}=-\infty,
\]
and that (ii) an entire function $f$ admits a Newton series expansion if it is of exponential type at most $\ln 2$, i.e.
\[
	\limsup_{r\to\infty}\frac{\ln^+M(r,f)}{r}\le\ln 2.
\]
This motivates the following construction.
\bigskip

As in Definition \ref{D:AF}, we let $F\subset\mathbb{C}[[X]]$ be the subspace defined by
	\[
		F=\Big\{\sum_k a_kX^k:\ 
		\lim_{n\to\infty}\frac{\ln |\sum_{k=n}^\infty(-1)^k k!\,a_k|}
		{\ln n}=-\infty\Big\}
	\]
and let $\mathcal{A}_F$ be the $\mathbb{C}$-algebra generated by $F$ and $\partial$ subject to the relation \linebreak $\partial X-X\partial -1$. Then $\mathcal{A}_F$ contains $\mathcal{A}$ as a subalgebra, and from now on we denote $\mathcal{A}_F$ by just $\mathcal{A}$. If $H\subset\mathcal{O}(\mathbb{C})$ denotes the subspace consisting of entire functions of exponential type at most $\ln 2$, then $H$ is endowed with the structure of a left $\mathcal{A}$-module by
	\begin{equation}\label{E:difference-A}
\begin{array}{l}
(Xf)(t)=tf(t-1);\\
(\partial f)(t)=f(t+1)-f(t)\hspace{1cm}\mbox{for all }f\in H\mbox{ and all }t\in\mathbb{C}.
\end{array}
	\end{equation}
Fix any $1$-periodic entire function $p$ (i.e. $p(t+1)=p(t)$ for all $t$), then we obtain a left $\mathcal{A}$-linear map
\begin{equation}\label{E:evaluation-A}
N:\mathcal{A}/\mathcal{A}\partial\stackrel{\cdot p}{\longrightarrow}Hp.
\end{equation}
In particular, when $p$ is the constant function $1$, the series
\[
\Big[\sum_k a_kX^k\Big]\cdot 1=\sum_k a_kx(x-1)\cdots(x-k+1)
\]
becomes the Newton series expansion of the entire function $N[\sum_ka_kX^k]$. Clearly we obtain a more general expansion when $p$ is chosen to be a non-constant $1$-periodic function.
%
\bigskip

\subsection{Residues} \label{SS:residues}

In this subsection we discuss the notion of residues. Recall that if $\mathcal{O}$ and $\Omega$ are the sheaves of analytic functions and one-forms on $\mathbb{C}^*$ respectively, then the classical residue is the map $\mathrm{Res}$ which fits into the exact sequence
\[
H^0(\mathbb{C}^\ast,\mathcal{O})\stackrel{d}{\to}H^0(\mathbb{C}^\ast,\Omega)\stackrel{\mathrm{Res}}{\longrightarrow}H^1(\mathbb{C}^\ast,\mathbb{C})\to 0,
\]
where the map $d$ is the exterior derivative, whose definition originates from the usual derivative $D$. In this article, however, the usual derivative $D$ is just a special interpretation of the symbol $\partial$ in the Weyl algebra, so the topological tool above would be lost if one attempts to interpret $\partial$ alternatively. Local cohomologies (see \cite{Hartshorne}) turns out to be the pure algebraic tool suitable for this purpose and $\mathrm{Res}$ is now a map which fits into the exact sequence
\[
H^1_0(\mathbb{C}[X])\stackrel{\partial\times}{\longrightarrow}H^1_0(\mathbb{C}[X])\stackrel{\mathrm{Res}}{\longrightarrow}H^2_0(\mathbb{C})\to 0.
\]
Here the local cohomology $H^1_0(\mathbb{C}[X])$ is the space of principal parts at $0$. Just for the narrow purpose in this article, the general theory of local cohomologies is not needed. The first task is to introduce the symbol $X^{-1}$ or $1/X$ so that the discussion of principal parts mentioned above is possible.
\medskip

\begin{definition}
Choose a $\mathbb{C}$-vector subspace $\mathbb{C}[X]\subset F\subset\mathbb{C}[[X]]$. We use the notation $\mathcal{A}_{0}$ to denote the $\mathbb{C}$-algebra generated by $F$, $\dfrac{1}{X}$ and $\partial$, subject to the relation $\partial X-X\partial-1$. The dependence on $F$ in the notation $\mathcal{A}_{0}$ is again suppressed for simplicity.
\end{definition}

In the rest of this section, a choice of the subspace $F$ is understood without further emphasis.
\medskip

\begin{lemma}
The cokernel of the map
\[
\mathcal{A}_{0}/\mathcal{A}_{0}\partial\stackrel{\partial\times}{\longrightarrow}\mathcal{A}_{0}/\mathcal{A}_{0}\partial
\]
has $\mathbb{C}$-dimension one and is spanned by $1/X$.
\end{lemma}
\medskip


\begin{proof} It is easy to establish by induction that the formula
	\[
		\partial X^n=nX^{n-1} 	\qquad \mod \mathcal{A}_0\partial
	\]holds for each  integer $n\in\mathbb{Z}$. In particular, 
	$X^{-1}$ has no ``anti-$\partial$".
\end{proof}		
\bigskip

\begin{definition}
Let $M$ be a left $\mathcal{A}$-module. Then we denote
\[
M_0=\mathcal{A}_{0}\otimes_{\mathcal{A}} M.
\]
\end{definition}
\medskip


\begin{corollary}
\label{Res1d}
Let $M$ be a left $\mathcal{A}$-module freely generated by a singleton over $F$. Then the cokernel of the map
\[
M_0\ \stackrel{\partial\times}{\longrightarrow}\ M_0
\]
is one dimensional.
\end{corollary}
\medskip

We now come to define the key map.

\begin{definition}
Let $M$ be a left $\mathcal{A}$-module freely generated by a singleton over $F$. Then a choice of an isomorphism
\[
\Res:M_0/\partial M_0\to\mathbb{C}
\]
is called a \textit{residue} map. In particular, $\Res f=0$ if and only if $f=\partial g$ for some $g\in M_0$ (i.e., $f$ has an ``anti-$\partial$").
\end{definition}

Corollary~\ref{Res1d} implies that any two such choices of the residue map differ just by a non-zero complex scalar multiple.
\bigskip

We see immediately that the extraction of Taylor coefficients by the residue map defined above is in general explicitly described in the following theorem.

\begin{theorem}[Cauchy integral formula]\label{T:CIFA}
Let $M$ be a left $\mathcal{A}$-module which is freely generated by a singleton $\{p\}$ over $F$. For each choice of residue map $\Res:M_0/\partial M_0\to\mathbb{C}$, there exists a scalar $c\in\mathbb{C}^*$ such that if $f\in M$ has the ``Taylor series expansion"
\[
f=\Big[\sum_ka_kX^k\Big]p,
\]
then
	\begin{equation}\label{E:residue_1}
	a_k=c\, \Res\Big(\dfrac{1}{X^{k+1}}f\Big), \hspace{1cm}\mbox{for every }k\in\mathbb{N}.
	\end{equation}
\end{theorem}

\subsection{Liouville's theorem for the difference operator $\Delta$}
\label{SS:DLiouville}

We emphasise that the development in \S\ref{SS:Weyl}--\ref{SS:residues} about residues is completely general where the symbol ``$\partial$" is up to interpretation. We finally come to our first version of Liouville's theorem by considering the left $\mathcal{A}$-module effected by the interpretation \eqref{E:difference-A}, i.e. $(\partial f)(x)=\Delta f(x)=f(x+1)-f(x)$. 

\bigskip
The following classical result of Whittaker states that for any given entire function $f$, we have $\Res f=0$. That is, an ``anti-difference" always exists:

\begin{theorem}[\cite{whittaker}(pp. 22--24)] If $f$ is an entire function, then there exists an entire function $g$ (of the same order as $f$) such that 
	\[
		f(x)=\partial g(x)=g(x+1)-g(x).
	\]
\end{theorem}
\medskip

\begin{theorem}[Difference Liouville\rq{s}  theorem] \label{T:liouville_1}
Let $f$ be an entire function. If there exists a $1$-periodic meromorphic function $p$ such that 
    \begin{enumerate}
        \item $\displaystyle\frac{f}{p}$ is entire and is of exponential type at most $\ln 2$, and
        \item there exists a sequence of positive real numbers $\{x_n\}$ diverging to $+\infty$ such that
			\[
				\frac{f}{p} (x_n+iy)=o(x_n)
			\]
        for every $y\in\mathbb{R}$ uniformly as $n\to\infty$,
    \end{enumerate}
    then
    $f$ is $1$-periodic.
\end{theorem}

We note that the Phragm\'en-Lindel\"of theorem does not apply to functions described in Theorem \ref{T:liouville_1}. In fact, these two theorems are independent of each other.

\subsubsection{Proof of difference Liouville\rq{s} theorem}

Theorem \ref{T:liouville_1} will be established in several steps below. 

\subsubsection*{Construction of a residue map}
Fix a $1$-periodic function $p$, and let $M=Hp$ be a left $\mathcal{A}$-module where $H\subset\mathcal{O}(\mathbb{C})$ is the subspace of entire functions consisting of exponential growth of type $\ln 2$ at most as defined  previously in \S\ref{SS:Difference}. For each $a>0$, one defines the residue as the following map
	\[
		\begin{array}{crcl}	
	\Res : & M_0/\partial M_0&\longrightarrow&\mathbb{C}\\
	&	f&\mapsto&\displaystyle\int_{a-i\infty}^{a+i\infty} \frac{f(x)}{p(x)}e^{-\cos(2\pi x)}\, dx,
\end{array}
	\]
where the improper integral is understood as its Cauchy principal value. Here the \textit{weight function} $e^{-\cos(2\pi x)}$ guarantees the convergence of the improper integral. We first show that this residue map is well-defined.
\medskip

\begin{lemma}\label{T:uniqueness} 	
If $f=\partial g$ for some $g\in M_0$, then  $\Res f =0$.
\end{lemma}

\medskip

\begin{proof} If $f(x)=g(x+1)-g(x)$, then
\begin{align*}
	\int_{a-i\infty}^{a+i\infty} \frac{f(x)}{p(x)}\, e^{-\cos(2\pi x)}\, dx 
	&= \lim_{T\to+\infty} \int_{a-iT}^{a+iT} \frac{f(x)}{p(x)}\, e^{-\cos(2\pi x)}\, dx\\
	&= \lim_{T\to+\infty} \int_{a-iT}^{a+iT} \frac{g(x+1)-g(x)}{p(x)}\, e^{-\cos(2\pi x)}\, dx \\
	&= \lim_{T\to+\infty} \left(\int_{a+1-iT}^{a+1+iT}-\int_{a-iT}^{a+iT}\right) \frac{g(x)}{p(x)}\, e^{-\cos(2\pi x)}\, dx\\
	&= \lim_{T\to+\infty}  \int_{\Gamma_T} \frac{g(x)}{p(x)}\, e^{-\cos(2\pi x)}\, dx \\
		&= 0
\end{align*}
by Cauchy's integral theorem, since the contour $\Gamma_T$ is the boundary of the rectangle with vertices $a+iT$, $a-iT$, $a+1-iT$ and $a+1+iT$, and the two improper (line) integrals
	\[
		\lim_{T\to+\infty}\int_{[a\pm iT,\, a+1\pm iT]} \frac{g(x)}{p(x)}e^{-\cos(2\pi x)}\, dx
	\]
	both vanish.
\end{proof}

Although the choice of the weight function $w(x) =e^{-\cos(2\pi x)}$ used in the above proof is not unique, Corollary \ref{Res1d} ensures that the residue map $\Res$ is, however, unique up to a constant multiple.

\begin{lemma}
\label{T:orderone}
Let $f$ be an entire function of exponential type at most $\ln 2$. If there exists a sequence $\{x_n\}$ of positive real numbers tending to $+\infty$ such that
	\[
		f(x_n+iy)=o(x_n)
	\]
for every $y\in\mathbb{R}$ uniformly as $n\to\infty$, then $f$ is a constant.
\end{lemma}

\begin{proof}
Let $H\subset\mathcal{O}(\mathbb{C})$ be the subspace consisting of entire functions of exponential type at most $\ln 2$, which is considered as a left $\mathcal{A}$-module as in \S\ref{SS:Difference}. Now $f\in H$ admits a ``Newton series expansion"
	\[
		f=\Big[\sum_{k=0}^\infty a_kX^k\Big]\cdot 1
	\]
via the map \eqref{E:evaluation-A}. Note that the sequence $\{x_n\}$ has either a subsequence $\{x_{n_j}\}$ such that $\cos(2\pi x_{n_j})\ge 0$ for all $j$, or a subsequence $\{x_{n_j}\}$ such that $\cos(2\pi x_{n_j})\le 0$ for all $j$. In the former case we choose the residue map to be
\[
	\Res f = \int_{a-i\infty}^{a+i\infty}f(t)\, e^{-\cos(2\pi t)}\,dt
\]
as before, while in the latter case we choose the residue map to be
\[
	\Res f = \int_{a-i\infty}^{a+i\infty}f(t)\, e^{\cos(2\pi t)}\,dt
\]
instead (Here $a$ can be any positive real number because of Cauchy's integral theorem). Then in both cases, we have 
	\[
		|e^{\mp\cos(2\pi t)}|=e^{\mp\cos(2\pi\Re t)\cosh(2\pi\Im t)}\le 1
	\] when $\Re t=x_{n_j}$.  
	So for every $k\ge 1$, we have by Theorem \ref{T:CIFA}
\begin{align*}
|a_k|&= \left|c\Res\Big(\dfrac{1}{X^{k+1}}f\Big)\right|\\
&=\left|c\int_{x_{n_j}-k-1-i\infty}^{x_{n_j}-k-1+i\infty}\dfrac{f(t+k+1)}{(t+1)(t+2)\cdots(t+k+1)}e^{\mp\cos(2\pi t)}\,dt\right|\\
 &\leq \dfrac{|c|}{(x_{n_j}-k)^{k+1}}\max_{\Re t=x_{n_j}}|f(t)|,
\end{align*}
which tends to $0$ as $j\to\infty$. Therefore $a_k=0$ for every $k\ge 1$.
\end{proof}

\subsubsection*{Completion of the proof of Theorem   \ref{T:liouville_1}} We simply apply Lemma \ref{T:orderone} to the function $f/p$ in order to complete the proof.\qed




\section{$q$-Liouville\rq{s} theorems} \label{S:q}

We briefly discuss how to modify the Weyl algebra \eqref{D:weyl-alg} into the setting of $q$-difference operator. In the whole \S\ref{S:q}, we let $q\in\mathbb{C}^*$ and suppose that $|q|\ne 1$.

\subsection{The $q$-Weyl algebra}

\begin{definition}
The \textit{$q$-Weyl algebra} $\mathcal{A}_q$ is the $\mathbb{C}$-algebra defined by
	\[
\mathcal{A}_q=\mathbb{C}\langle X,\partial\rangle /\langle \partial X-qX\partial-1\rangle.
	\]
\end{definition}

\begin{remark}
One can recover the results concerning $\mathcal{A}$ from \S\ref{S:main} by taking limit $q\to 1$ in those results concerning $\mathcal{A}_q$.
\end{remark}

Fix $L\in\mathcal{A}_q$, and denote by $\mathcal{A}_qL$ the left ideal generated by $L$. Then the quotient
\[
\mathcal{A}_q/\mathcal{A}_qL
\]
is a left $\mathcal{A}_q$-module.


\begin{definition}\label{D:AqF}
Fix a $\mathbb{C}$-vector subspace $\mathbb{C}[X]\subset F\subset\mathbb{C}[[X]]$ of the space of formal power series in $X$.  We use the notation $\mathcal{A}_{q,\, F}$ to denote the ring extension of $\mathcal{A}_q$ generated by $F$ and $\partial$ as a $\mathbb{C}$-algebra, subject to the relation $\partial X-qX\partial-1$. We suppress the emphasis of the dependence of $\mathcal{A}_{q,\, F}$ on $F$ in later usages when there is no ambiguity.
\end{definition}
\bigskip

\subsubsection{Applications to Jackson\rq{s} $q$-difference operator}
\label{eg:q-derivative}

Suppose that $0<|q|<1$. As in Definition \ref{D:AqF}, we let $F\subset\mathbb{C}[[X]]$ be the subspace defined by
	\begin{equation}\label{E:F}
		F=\Big\{\sum_k a_kX^k:\limsup_n|a_n|^{1/n}<1\Big\},
	\end{equation}
and let $\mathcal{A}_{q,F}$ be the $\mathbb{C}$-algebra generated by $F$ and $\partial$ subject to the relation \linebreak $\partial X-qX\partial -1$. Then $\mathcal{A}_{q,F}$ contains $\mathcal{A}_q$ as a subalgebra, and from now on we denote $\mathcal{A}_{q,F}$ by just $\mathcal{A}_q$. Consider the subspace $H\subset\mathcal{O}(\mathbb{C}^*)$ consisting of all analytic functions defined on the punctured complex plane which have $q$-exponential growth of order $<\ln |q^{-1}|$ \cite{Annaby-Mansour,Ramis}. Then $H$ is endowed with the structure of a left $\mathcal{A}_q$-module by
	\begin{equation}\label{E:q-difference-A1}
\begin{array}{l}
{(Xf)(t)=(t-1)f(qt)};\\
(\partial f)(t)=\dfrac{f(t)-f(t/q)}{t-t/q}\hspace{1cm}\mbox{for all }f\in H\mbox{ and all }t\in\mathbb{C}^*.
\end{array}
	\end{equation}
Fix any ``periodic function" $p$ (i.e. $p(qt)=p(t)$ for all $t$, so that $p$ must in fact be a constant), then we obtain a left $\mathcal{A}_q$-linear map
	\begin{equation}\label{E:evaluation-map}
N_q:\mathcal{A}_q/\mathcal{A}_q\partial\stackrel{\cdot p}{\longrightarrow}H.
	\end{equation}
Similar to the previous application in \S\ref{SS:Difference}, take $p$ to be the constant function $1$, then we recover
\[
\sum_k(-1)^ka_k(1-x)(1-qx)\cdots(1-q^{k-1}x)
\]
as the series expansion of the analytic function $N_q[\sum_ka_kX^k]$ in the polynomial base $\{(x;q)_n\}$ \cite{Jackson,Annaby-Mansour,Ismail}.

On the other hand, if $|q|>1$, then in Definition \ref{D:AqF} we let $F\subset\mathbb{C}[[X]]$ be the subspace defined by
	\[
		F=\Big\{\sum_k a_kX^k:\limsup_n|a_n|^{\frac{2}{n(n-1)}}<|q^{-1}|\Big\}.
	\]
Then the same left $\mathcal{A}_q$-module structure as in \eqref{E:q-difference-A1} and \eqref{E:evaluation-map} can be equipped on the whole space $\mathcal{O}(\mathbb{C})$ of all entire functions, and we obtain the same kind of series expansions in the polynomial base $\{(x;q)_n\}$ \cite{Annaby-Mansour}.

\bigskip
In whichever case for $|q|$ above, if instead of \eqref{E:q-difference-A1}, we endow $\mathcal{O}(\mathbb{C}^*)$ with a slightly different left $\mathcal{A}_q$-module structure
\begin{equation}\label{E:q-difference-A2}
\begin{array}{l}
{(Xf)(t)=tf(qt)};\\
(\partial f)(t)=\dfrac{f(t)-f(t/q)}{t-t/q}\hspace{1cm}\mbox{for all }f\in\mathcal{O}(\mathbb{C}^*)\mbox{ and all }t\in\mathbb{C}^*,
\end{array}
	\end{equation}
and choose the map
\begin{equation}\label{E:evaluation-map-2}
  \tilde{N_q}:\mathcal{A}_q/\mathcal{A}_q\partial\stackrel{\cdot 1}{\longrightarrow}\mathcal{O}(\mathbb{C}^*).
	\end{equation}
instead of \eqref{E:evaluation-map} (note that the definition of the scalar multiplications in \eqref{E:evaluation-map} and \eqref{E:evaluation-map-2} are different), then it gives a series expansion of the analytic function $\tilde{N_q}[\sum_ka_kX^k]$ which assumes the form
	\[
		\sum_k a_k q^{\binom{k}{2}}x^k.
	\]
This includes the well-known \textit{basic hypergeometric series} \cite{Gasper_Rahman}. 
\bigskip

Once we have fixed the basic notations, the following statements which are parallel to those from \S\ref{S:main} become self-evident.
\bigskip

\subsection{$q$-Residues}
\begin{definition}
Choose a $\mathbb{C}$-vector subspace $\mathbb{C}[X]\subset F\subset\mathbb{C}[[X]]$. We use the notation $\mathcal{A}_{q,0}$ to denote the $\mathbb{C}$-algebra generated by $F$, $\dfrac{1}{X}$ and $\partial$, subject to the relation $\partial X-qX\partial-1$. The dependence on $F$ in the notation $\mathcal{A}_{q,0}$ is again suppressed for simplicity.
\end{definition}

In the rest of this section, a choice of $F$ is understood without further emphasis.
\medskip

\begin{lemma}
The cokernel of the map
\[
\mathcal{A}_{q,0}/\mathcal{A}_{q,0}\partial\stackrel{\partial\times}{\longrightarrow}\mathcal{A}_{q,0}/\mathcal{A}_{q,0}\partial
\]
has $\mathbb{C}$-dimension one and is spanned by $1/X$.
\end{lemma}
\begin{proof}
One easily proves by induction that
\[
\partial X^n=\dfrac{1-q^n}{1-q}X^{n-1}\hspace{1cm}\mod\;\mathcal{A}_{q,0}\partial
\]
for each $n\in\mathbb{Z}$.
\end{proof}

\begin{definition}
Let $M$ be a left $\mathcal{A}_q$-module. Then we denote
\[
M_0=\mathcal{A}_{q,0}\otimes_{\mathcal{A}_q} M.
\]
\end{definition}

\begin{corollary}
Let $M$ be a left $\mathcal{A}_q$-module freely generated by a singleton over $F$. Then the cokernel of
\[
M_0\stackrel{\partial\times}{\longrightarrow}M_0
\]
is one dimensional.
\end{corollary}

\begin{definition}
Let $M$ be a left $\mathcal{A}_q$-module freely generated by a singleton over $F$. Then a choice of an isomorphism
\[
\mathrm{Res}:M_0/\partial M_0\longrightarrow\mathbb{C}
\]
is called a \textit{residue} map.
\end{definition}

\begin{theorem}[Cauchy integral formula]\label{T:CIFAq}
Let $M$ be a left $\mathcal{A}_q$-module which is freely generated by a singleton $\{p\}$ over $F$. For each choice of residue map $\Res:M_0/\partial M_0\to\mathbb{C}$, there exists a scalar $c\in\mathbb{C}^*$ such that if $f\in M$ has the ``Taylor series expansion"
\[
f=\Big[\sum_ka_kX^k\Big]p,
\]
then
		\begin{equation}\label{E:residue_2}
	a_k=c\, \Res\Big(\dfrac{1}{X^{k+1}}f\Big)\hspace{1cm}\mbox{for every }k\in\mathbb{N}.
	\end{equation}
\end{theorem}
\bigskip

\subsection{$q$-Liouville's theorems}

Suppose that $0<|q|<1$, let $F\subset\mathbb{C}[[X]]$ be the subspace as defined in \eqref{E:F} and let $H\subset\mathcal{O}(\mathbb{C}^*)$ be the first left $\mathcal{A}_q$-module in \S\ref{eg:q-derivative}, defined by \eqref{E:q-difference-A1} and \eqref{E:evaluation-map}.

\subsubsection*{Construction of a residue map}

For each $r>1$, one defines the residue as
	\[
	\begin{array}{rrcl}
	\Res:&H_0/\partial H_0&\longrightarrow &\mathbb{C}\\
	&f&\mapsto&\dfrac{1}{2\pi i}\displaystyle\int_{\partial D(0,r)} f(x)\, dx.
	\end{array}
\]
\medskip

It is easy to see that this residue map is also well-defined.

\begin{lemma}\label{T:uniqueness-q} 	
If $f=\partial g$ for some $g\in H_0$, then  $\Res f =0$.
\end{lemma}

\begin{proof} If $\displaystyle f(x)=\frac{g(x)-g(x/q)}{x-x/q}$, then
\begin{align*}
	\int_{\partial D(0,r)}{f(x)\,dx}
	&= \int_{\partial D(0,r)}{\frac{g(x)-g(x/q)}{x-x/q}\,dx}\\
	&= \left(\int_{\partial D(0,r)}-\int_{\partial D(0,r/|q|)}\right) \frac{g(x)}{x-x/q}\, dx\\
	&= 0
\end{align*}
by Cauchy's integral theorem, since $g$ is analytic on the annulus bounded by the circles $\partial D(0,r)$ and $\partial D(0,r/|q|)$.
\end{proof}

\bigskip

\begin{theorem}\label{T:q-liouville}
Let $f$ be a holomorphic function defined on $\mathbb{C}^*$, which has $q$-exponential growth of order $<\ln |q^{-1}|$ \cite{Annaby-Mansour,Ramis}.
If
\[
|f(t)|=o(|t|)
\]
as $t\to\infty$, then $f$ is a constant.
\end{theorem}

\begin{proof}
$H$ is a left $\mathcal{A}_q$-module which is freely generated by $\{1\}$ over $F$. By assumption, the given function $f\in H$ has a $q$-Taylor series expansion
\[
f=\Big[\sum_{k=0}^\infty a_kX^k\Big]\cdot 1
\]
via the map \eqref{E:evaluation-map}. Then for each $k\geq 1$ and $r>0$, we have by Theorem \ref{T:CIFAq}
\begin{align*}
|a_k|&=\left| \Res\Big(\dfrac{1}{X^{k+1}}f\Big)\right|\\
 &=\left|\frac{1}{2\pi i}\int_{\partial D(0,r)}\dfrac{f(t/q^{k+1})}{(t/q-1)(t/q^2-1)\cdots(t/q^{k+1}-1)}\,dt\right|\\
 &\le \frac{1}{2\pi}\frac{2\pi rM(r/|q|^{k+1},f)}{(r/|q|-1)(r/|q|^2-1)\cdots(r/|q|^{k+1}-1)}\\
 &=o(r^{1-k})
\end{align*}
which tends to $0$ as $r\to\infty$. Therefore $a_k=0$ for every $k\ge 1$.
\end{proof}
\bigskip

\begin{remark} If one considered the left $\mathcal{A}_q$-module structure defined by \eqref{E:q-difference-A2} and \eqref{E:evaluation-map-2} in this subsection, then the classical Liouville's theorem would be obtained instead of Theorem \ref{T:q-liouville}. Although the same formula $a_k=c\,\Res\Big(\dfrac{1}{X^{k+1}}f\Big)$ works for both left $\mathcal{A}_q$-module structures, its interpretation via \eqref{E:q-difference-A2} and \eqref{E:evaluation-map-2} is different from that via \eqref{E:q-difference-A1} and \eqref{E:evaluation-map} currently being considered.
Also note that although one can obtain the classical Liouville's theorem this way, it is still conceptually different from obtaining the same theorem using the left $\mathcal{A}$-module structure defined by the differential operator, which will be done in \S\ref{S:classical}.
\end{remark}

\bigskip

\section{Classical Liouville's theorem}\label{S:classical}

We now return to the Weyl algebra $\mathcal{A}$ introduced in \S\ref{S:main}. 

Let $U\subset\mathbb{C}$ be an open subset. We denote by $\mathcal{O}(U)$ the space of analytic functions defined on $U$. Then $\mathcal{O}(U)$ is endowed with the structure of a left $\mathcal{A}$-module by
\[
\begin{array}{l}
(Xf)(t)=tf(t);\\
(\partial f)(t)=f'(t)\hspace{1cm}\mbox{for all }f\in\mathcal{O}(U)\mbox{ and all }t\in U.
\end{array}
\]

Let $F\subset\mathbb{C}[[X]]$ be the subspace defined by
\[
F=\Big\{\sum_k a_kX^k:\lim_{n\to\infty}|a_n|^{1/n}=0\Big\},
\]
and $\mathcal{A}$ be the algebra generated by $F$ and $\partial$ (the dependence of $F$ has again been suppressed in the notation). 
The map
\[
\mathcal{A}/\mathcal{A}\partial\stackrel{\cdot 1}{\longrightarrow}\mathcal{O}(\mathbb{C})
\]
becomes an isomorphism of left $\mathcal{A}$-modules. It is the usual Maclaurin series expansion of entire functions.
\bigskip

\subsubsection*{Construction of a residue map}

Let $\mathcal{O}(\mathbb{C})$ be the left $\mathcal{A}$-module above. For each $r>0$, one defines the residue as
	\[
	\begin{array}{rrcl}
	\Res:&\mathcal{O}(\mathbb{C})_0/\partial\mathcal{O}(\mathbb{C})_0&\longrightarrow &\mathbb{C}\\
	&f&\mapsto&\dfrac{1}{2\pi i}\displaystyle\int_{\partial D(0,r)} f(x)dx.
	\end{array}
\]

Hence we obtain

\begin{theorem}[Liouville]
Let $f\in\mathcal{O}(\mathbb{C})$ be an entire function such that
\[
|f(t)|=o(|t|)\hspace{1cm}\mbox{as }t\to\infty.
\]
Then $f$ is a constant.
\end{theorem}
\begin{proof}
$\mathcal{O}(\mathbb{C})$ is a left $\mathcal{A}$-module which is freely generated by $\{1\}$ over $F$. Expand the given entire function $f$ into a power series
\[
f=\Big[\sum_{k=0}^\infty a_k X^k\Big]\cdot 1.
\]
By the hypothesis, $|f(t)|\le M_{|t|}|t|$ for some $M_r$ such that $\lim_{r\to+\infty}M_r=0$. Then for each $k\geq 1$ and $r>0$, we have

\begin{align*}
|a_k|&=\left|\mathrm{Res}\Big(\dfrac{1}{X^{k+1}}f\Big)\right|\\
&=\dfrac{1}{2\pi}\left|\int_{\partial D(0,r)}\dfrac{f(t)dt}{t^{k+1}}\right|\\
&\leq\dfrac{1}{2\pi}2\pi r\dfrac{M_r r}{r^{k+1}}
\end{align*}
which tends to $0$ as $r\to+\infty$. Therefore $a_k=0$ for every $k\ge 1$.
\end{proof}


\section{Concluding remarks}
\label{S:discussion}

\subsubsection*{Discussion of main results}
In this article, we have found different versions of Liouville's theorems corresponding to various difference operators by adopting an algebraic approach to understanding series expansions in complex function theory. It is found that various kinds of series expansions of analytic or meromorphic functions can all be unified as different interpretations of the same ``symbolic" power series, which is an element in $\overline{\mathcal{A}/\mathcal{A}\partial}$. This new point of view has led to a unified \textit{residue theory}, which is the main ingredient in this article towards obtaining different kinds of Liouville's theorems obtained.

Although the residue map \eqref{E:residue_1} and the Cauchy integral formula \eqref{E:residue_2} are in an abstract algebraic setting, it turns out that for the particular cases of difference or differential operators that we consider, they can be realized using contour integrals.  In the consideration of the (forward) difference operator $\Delta f(x)=f(x+1)-f(x)$, the residue map is realized as a Barnes-type contour integral;
and in the cases of  Jackson's $q$-difference operator $D_qf(x)=[f(x)-f(x/q)]/(x-x/q)$ and the usual differential operator, the residue map is realized as a Cauchy-type contour integral.
These machineries therefore lead to the different analogues of Liouville's theorem obtained in this article. We emphasize that the boundedness conditions in these new analogues of Liouville's theorem may not be the same as the traditional ``boundedness".  Firstly, In  the setting of (forward) difference operator: entire functions $f$ that grow slower than a specific growth rate in the positive real direction reduce to a $1$-periodic function, i.e. $\Delta f=0$. Obviously, periodic functions  need not be  bounded in the classical sense. This conclusion also agrees with the Little Picard theorem for difference operator $\Delta$ discovered by Halburd and Korhonen \cite{Halburd-Korhonen}. Secondly, for Jackson's $q$-difference operator, holomorphic functions on $\mathbb{C}^*$ that grow slower than a specific growth rate reduce to an $f$ such that $D_qf=0$, i.e. a constant. Finally, for differential operator, we recover the classical Liouville's theorem: an entire function that is bounded (``traditionally") reduces to a constant.

\subsubsection*{Application to de Branges\rq{} space of entire functions}
Let $E$ be an entire function which satisfies
	\[
		|E(\bar{z})|<|E(z)|
	\]
	for all $z$ with $\Im(z)>0$. De Branges \cite{deBranges} considered the set  $\mathscr{H}(E)$ of entire functions $F$ such that
	\[
		\|F\|^2=\int_\mathbb{R} \Big|\frac{F(t)}{E(t)}\Big|^2\, dt<\infty
	\]and
	\[
		|F(z)|^2\le \|F\|^2 \frac{|E(z)|^2-E(\bar{z})|^2}{2\pi i (\bar{z}-z)}
	\]for all $z\in\mathbb{C}$. He showed that $\mathscr{H}(E)$ is a Hilbert space that enjoys special properties (see \cite[(H1--3)]{deBranges}) that generalise the classical Fourier series theory. He also showed that if $\mathscr{H}(E)$ is $1$-periodic (i.e. every function in $\mathscr{H}(E)$ is $1$-periodic), then one has the factorization
	\begin{equation}\label{E:factor}
		E(z)=F(z)\, G(z),
	\end{equation}
	where $F$ is an entire function which has only real zeros and satisfies $F(z-1)=\pm F(z)$, and $G$ is an entire function of exponential type with \textit{no} real zeros, such that $\mathscr{H}(G)$ is $1$-periodic. Our difference Liouville's Theorem (Theorem \ref{T:liouville_1}) shows that if we assume further that $G$ is ``bounded'' in the difference sense, (i.e., of exponential type $<\ln 2$ and grows slower than $\Re z$ in the direction parallel to the positive real axis), then $G$ must reduce to constant and $E$ itself must therefore be periodic according to  \eqref{E:factor}.

\subsubsection*{$D$-module approach  to solutions of difference/differential equations}

The Weyl algebraic point of view appears to be useful in studying almost all aspects of complex function theory, whenever differential and/or difference operators are involved \cite{CCT}. As an example, each family of classical hypergeometric orthogonal polynomials or basic hypergeometric orthogonal polynomials in the Askey scheme \cite{Ismail} can be presented as polynomial solutions to differential or ($q$-)difference equations. One can then make use of the Weyl algebraic machinery introduced in this article to study these families of polynomials as elements in left $\mathcal{A}$-modules $\mathcal{A}/\mathcal{A}L$ for different choices of $L\in\mathcal{A}$. At this point, we should mention that our $D$-module approach is different from the more well-known \textit{umbral calculus} approach advocated by the Rota school \cite{Roman} since the 1970's. 

Specifically, from the Weyl algebraic point of view, one realizes that the Charlier polynomial $C_n(x;a)$ is in fact a ``difference analogue" of the simple binomial \linebreak $(1-x/a)^n$. To see this, one refers to the Rodrigues formula for Charlier polynomials \cite[Chap. 5]{Beals_Wong_2016}, \cite{Koekoek}, which is
\[
	\frac{a^x}{x!}C_n(x;a)=\nabla^n\frac{a^x}{x!},
\]
where $\nabla$ is the backward shift operator defined by $\nabla f(x)=f(x)-f(x-1)$. One can then rearrange this to become
\[
	C_n(x;a) = \frac{x!}{a^x}\nabla^n\frac{a^x}{x!} = \left(\frac{x!}{a^x}\nabla\frac{a^x}{x!}\right)^n
\]
One notes that for any function $f$,
\[
	\left(\frac{x!}{a^x}\nabla\frac{a^x}{x!}\right)f(x)=f(x)-\frac{x}{a}f(x-1)=\left(1-\frac{X}{a}\right)f(x),
\]
if $X$ is interpreted using the left $\mathcal{A}$-module structure \eqref{E:difference-A} regarding the (forward) difference operator. As a result, the Charlier polynomial $C_n(x;a)$ is just the image of 
	the map
		\[
			\left(1-\frac{X}{a}\right)^n
		\] under a scalar multiplication to the constant function $1$, similar to \eqref{E:evaluation-A}:
\[
	C_n(x;a)=\sum_{k=0}^n{\binom{n}{k}(-a)^{-k}x(x-1)\cdots(x-k+1)}.
\]
The orthogonality of these polynomials can also be formulated as a pairing making use of the residue map defined in this article, which is now the obstruction of having ``anti-$(X-a)$" and is usually expressed as a sum rather than an integral. Their generating functions can also be derived via $D$-modules, making use of a holonomic system of PDEs. It turns out that the D-module approach hinted in this paper not only leads to an alternative organisation of classical theories of special functions, but also gives new discoveries of special functions, which will be pursued in a forthcoming work.

\bibliographystyle{amsalpha}

\begin{thebibliography}{10}
\bibitem{Annaby-Mansour} M. H. Annaby and Z. S. Mansour, \textit{$q$-Taylor and interpolation series for Jackson $q$-difference operators}, J. Math. Anal. Appl. \textbf{344} (2008), 472--483.
\bibitem{BHMK} D. C. Barnett, R. G. Halburd, W. Morgan and R. J. Korhonen, \textit{Nevanlinna theory for the $q$-difference operator and meromorphic solutions of $q$-difference equations}, Proc. R. Soc. Edinb., Sect. A, Math. \textbf{137} (3) (2007), 457--474.
\bibitem{Beals_Wong_2016} R. Beals and R. Wong, \textit{Special Functions and Orthogonal Polynomials}, Cambridge Studies in Advanced Mathematics 153, Cambridge University Press, Cambridge, 2016.
\bibitem{Boole} G. Boole, \textit{On a general method in analysis}, Philos. Trans. R. Soc. London \textbf{134} (1844), 225--282.
\bibitem{Cheng-Chiang} K. H. Cheng and Y. M. Chiang, \textit{Nevanlinna theory of the Wilson divided-difference operator}, Ann. Acad. Sci. Fenn. Math. \textbf{42} (1) (2017), 175--209.
\bibitem{CCT} Y. M. Chiang, A. Ching and C. Y. Tsang, \textit{Series solutions of linear ODEs by Newton-Raphson method on quotient $D$-modules}, submitted, arXiv: 2111.05103.
\bibitem{Chiang-Feng3} Y. M. Chiang and S. J. Feng, \textit{Nevanlinna theory of the Askey-Wilson divided difference operator}, Adv. Math \textbf{329} (2018), 217--272.
\bibitem{Chu_Vu_2003} C.-H. Chu and T. G. Vu, \textit{A Liouville theorem for matrix-valued harmonic functions on nilpotent groups}, Bull. Lond. Math. Soc. \textbf{35} (5), (2003), 651--658.
\bibitem{deBranges} L. de Branges, \textit{Homogeneous and periodic spaces of entire functions}, {Duke Math. J.}, \textbf{29} (1962), 203--224.
\bibitem{deBranges_1968} L. de Branges, \textit{Hilbert Spaces of Entire Functions}, Prentice-Hall, Englewood Cliffs, 1968.
\bibitem{Etingof} P. Etingof, O. Golberg, S. Hensel, T. Liu, A. Schwendner, D. Vaintrob and E. Yudovina, \textit{Introduction to Representation Theory. With historical interludes by Slava Gerovitch}, Student Mathematical Library 59, American Mathematical Society, Providence, 2011.
\bibitem{Farina} A. Farina, \textit{Liouville-type theorems for elliptic problems}, in Handbook of Differential Equations: Stationary Partial Differential Equations. Vol. IV, 61--116, Elsevier/North-Holland, Amsterdam, 2007.
\bibitem{Gasper_Rahman} G. Gasper and M. Rahman, \textit{Basic hypergeometric series}, Encyclopedia of Mathematics and its Applications 35, Cambridge University Press, Cambridge, 1990.
\bibitem{Gelfond} A. O. Gelfond, \textit{Calculus of Finite Differences}, Hindustan Publishing Corp., Delhi, 1971.
\bibitem{Halburd-Korhonen} R. G. Halburd and R. J. Korhonen, \textit{Nevanlinna theory for the difference operator}, Ann. Acad. Sci. Fenn. Math. \textbf{31} (2) (2006), 463--478.
\bibitem{Hartshorne} R. Hartshorne, \textit{Local cohomology: a seminar given by A. Grothendieck, Harvard University, Fall 1961. Notes by R. Hartshorne}, Lecture Notes in Mathematics 41, Springer-Verlag, Berlin-Heidelberg-New York, 1967.
\bibitem{Ismail} M. E. H. Ismail, \textit{Classical and Quantum Orthogonal Polynomials in One Variable}, Encyclopedia of Mathematics and its Applications 98, Cambridge University Press, Cambridge, 2005.
\bibitem{Jackson} F. H. Jackson, \textit{$q$-form of Taylor's theorem}, Messenger Math. \textbf{39} (1909), 62--64.
\bibitem{Koekoek} R. Koekoek, P. A. Lesky and R. F. Swarttouw, \textit{Hypergeometric Orthogonal Polynomials and Their $q$-analogues}, Springer Monographs in Mathematics, Springer, Berlin, 2010.
\bibitem{Koppelman} E. Koppelman, \textit{The calculus of operations and the rise of abstract algebra}, Arch. Hist. Exact Sci. \textbf{8} (1971), 155--242.
\bibitem{Laita} L. M. Laita, \textit{The influence of Boole's search for a universal method in analysis on the creation of his logic}, Ann. Sci. \textbf{34} (1977), 163--176.
\bibitem{Mansour_Schork} T. Mansour and M. Schork, \textit{Commutation Relations, Normal Ordering, and Stirling Numbers}, Discrete Mathematics and its Applications, CRC Press, Boca Raton, 2016.
\bibitem{Noerlund} N. E. N\"{o}rlund, \textit{Le\c{c}ons sur les s\'{e}ries d'interpolation}, Gauthier-Villars, Paris, 1926.
\bibitem{Ramis} J.-P. Ramis, \textit{About the growth of entire functions solutions of linear algebraic $q$-difference equations}, Ann. Fac. Sci. Toulouse Math. \textbf{1} (1) (1992), 53--94.
\bibitem{Roman} S. Roman, \textit{The Umbral Calculus}, Pure and Applied Mathematics 111, Academic Press, London, 1984.
\bibitem{Synowiec} J. A. Synowiec, \textit{Some highlights in the development of algebraic analysis}, Algebraic analysis and related topics, Warsaw, 11--46, Banach Cent. Publ., Warsaw, 2000.
\bibitem{whittaker} J. M. Whittaker, \textit{Interpolatory Function Theory}, Cambridge University Press, Cambridge, 1935.
\end{thebibliography}

\smallskip
\smallskip
\smallskip
\smallskip
\smallskip
\smallskip
\smallskip

\begin{minipage}{0.31\textwidth}
\begin{flushleft}
\begin{footnotesize}
Kam Hang Cheng\\ Department of Mathematics, \\ The Hong Kong University \\ of Science and Technology, \\ Clear Water Bay, \\ Kowloon, \\ Hong Kong SAR. \\
henry.cheng@family.ust.hk
\end{footnotesize}
\end{flushleft}
\end{minipage}
\noindent\begin{minipage}{0.31\textwidth}
\begin{flushleft}
\begin{footnotesize}
Yik Man Chiang\\ Department of Mathematics, \\ The Hong Kong University \\ of Science and Technology, \\ Clear Water Bay, \\ Kowloon, \\ Hong Kong SAR. \\
machiang@ust.hk
\end{footnotesize}
\end{flushleft} 
\end{minipage}
\begin{minipage}{0.31\textwidth}
\begin{flushleft}
\begin{footnotesize}
Avery Ching\\ Department of Statistics, \\ The University of Warwick, \\ Coventry, \\ CV4 7AL, \\ United Kingdom. \\
Avery.Ching@warwick.ac.uk
\end{footnotesize}
\end{flushleft} 
\end{minipage}
\end{document}